\documentclass{article}
\usepackage{fancyhdr,graphicx}
\usepackage{amsfonts}
\usepackage{amssymb}
\usepackage{amsthm}
\usepackage{newlfont}
\usepackage{amsmath}
\usepackage{multicol}
\usepackage{subfigure}

\newtheorem{theorem}{Theorem}[section]
\newtheorem{lemma}[theorem]{Lemma}
\newtheorem{remark}[theorem]{Remark}
\newtheorem{corollary}[theorem]{Corollary}
\newtheorem{proposition}[theorem]{Proposition}
\newtheorem{definition}[theorem]{Definition}

\newtheorem{example}[theorem]{Example}
\newtheorem{problem}[theorem]{Problem}
\graphicspath{{figures/}}

\usepackage{mathrsfs}
\usepackage{subfig}
\usepackage{picinpar}
\begin{document}
	
	\title
	{The number of quasi-trees of bouquets with exactly one non-orientable loop}
	
\author{Qingying Deng\\
		{\small School of Mathematics and Computational Science, Xiangtan University, P. R. China}\\
		Xian'an Jin\\
		{\small School of Mathematical Sciences, Xiamen University, P. R. China}\\
		Qi Yan\footnote{Corresponding author.}\\
		{\small School of Mathematics and Statistics, Lanzhou University, P. R. China}\\
		{\small{Email: qingying@xtu.edu.cn (Q. Deng), xajin@xmu.edu.cn (X. Jin), yanq@lzu.edu.cn (Q. Yan)}}
	}
	\date{}
	
	\maketitle
	
	\begin{abstract}
Recently, Merino extended the classical relation between the $2n$-th Fibonacci number and the number of spanning trees of the $n$-fan graph to ribbon graphs, and established a relation between the $n$-associated Mersenne number and the number of quasi-trees of the $n$-wheel ribbon graph. Moreover, Merino posed a problem of finding the Lucas numbers as the number of spanning quasi-trees of a family of ribbon graphs. In this paper, we solve the problem and give the Matrix-Quasi-tree Theorem for a bouquet with exactly one non-orientable loop. Furthermore, this theorem is used to verify that the number of quasi-trees of some classes of bouquets is  closely related to the Fibonacci and Lucas numbers. We also give alternative proofs of the number of quasi-trees of these bouquets by using the deletion-contraction
relations of ribbon graphs.
	\end{abstract}
	
	$\mathbf{keywords:}$ Quasi-tree, Matrix-Quasi-tree Theorem, Fibonacci number, Lucas number, Delta-matroid.

	\section{Introduction}
	The relation between the number of spanning trees of fans and wheels with the Fibonacci and Lucas numbers seems to fascinate mathematicians in combinatorics, see \cite{Sedl67,Sedl69}. That spanning quasi-trees play the same role for ribbon graphs as trees for abstract graphs is described in \cite{ChunJCTA,Chun}. More recently, Merino \cite{Merino} related the Fibonacci and associated Mersenne numbers with the number of spanning quasi-trees of two families of ribbon graphs.

We denote the $n$-th Fibonacci and Lucas numbers by $f_{n}$ and $\ell_{n}$, respectively. The sequence of Fibonacci and Lucas numbers share the same recursive formula, $f_{n} =f_{n-1} + f_{n-2}$ and $\ell_{n} =\ell_{n-1} + \ell_{n-2}$. However, the initial values are different. The first two Fibonacci numbers are $f_{1} = 1$ and $f_{2} = 1$, while the first two Lucas numbers are $\ell_{1} =1$ and $\ell_{2} =3$. Note that the relation between $f_{n}$ and $\ell_{n}$ is: $f_{n}+f_{n-2}=\ell_{n-1}$ when $n\geq 3$.
The sequence of associated Mersenne numbers $a_{n}$ was first defined in \cite{Haselgrove} as the integer sequence such that $a_{1} = 1$, $a_{2} = 1$ and $a_{n} =a_{n-1} + a_{n-2}+1-(-1)^{n}$. The relation between $a_{n}$ and $\ell_{n}$ is: $a_{n} =\ell_{n}-1-(-1)^{n}$.	
	
	

Let $\mathbb{F}_n$ be the bouquet with the signed rotation
$$(1,2,1,3,2,4,3,\cdots,i, i-1,i+1,i,\cdots,n-1,n-2,n,n-1,n),$$ whose chord diagram $D(\mathbb{F}_n)$ consists of the pairs $$\{(1,3),(2,5),(4,7),\cdots,(2n-4,2n-1),(2n-2,2n)\}.$$ The corresponding intersection graph of $\mathbb{F}_n$ is the $n$-path $P_n$.

Let $\mathbb{W}_n~(n\geq 3)$ be the bouquet with the signed rotation
$$(1,n,2,1,3,2,4,3,\cdots,i, i-1,i+1,i,\cdots,n-1,n-2,n,n-1),$$ whose chord diagram $D(\mathbb{W}_n)$ consists of the pairs $$\{(1,4),(3,6),(5,8),\cdots,(2n-3,2n),(2n-1,2)\}.$$
The corresponding intersection graph of $\mathbb{W}_n$ is the $n$-cycle $C_n$.

The number of quasi-trees of $\mathbb{F}_n$ and $\mathbb{W}_n$ were obtained from the recent work of Merino \cite{Merino} by using the Matrix-Quasi-tree theorem \cite{Bouchet87,Lauri,Macris,Merino2023}.

\begin{theorem}[\cite{Bouchet87,Lauri,Macris,Merino2023}]$($Matrix-Quasi-tree Theorem$)$
	Given an orientable bouquet $B$ with $n$ edges, the number of quasi-trees of $B$ equals $\det(I_{n} + A(B))$, where $A(B)$ is the intersection matrix of $B$ $($see Section 4$)$.
\end{theorem}

\begin{theorem}[\cite{Merino}]
The numbers of quasi-trees of $\mathbb{F}_{m}~(m\geq 0)$ and $\mathbb{W}_{n}~(n\geq 3)$ are equal to $(m+1)$-th Fibonacci number $f_{m+1}$ and $n$-th associated Mersenne number $a_{n}$, respectively.
\end{theorem}

Moreover, Merino \cite{Merino} posed the following problem:
\begin{problem}
Find the Lucas numbers as the number of spanning quasi-trees of a family of ribbon graphs.
\end{problem}

In this paper, we solve the problem by establishing a relation between the $(n-1)$-th Lucas number and the number of quasi-trees of a bouquet with
$n$ edges whose signed rotation is $$(1,2,3,2,1,4,3,5,4,\cdots,i, i-1,i+1,i,\cdots,n-1,n-2,n,n-1,n).$$

It is worth noticing that the edge set of a ribbon graph together with its spanning quasi-trees form a delta-matroid.  The connection between ribbon graph theory and delta-matroid theory, as well as the philosophy that delta-matroid theory generalises topological graph theory, is due to Chun, Moffatt, Noble, and Rueckriemen \cite{ChunJCTA, Chun}. We will give the Matrix-Quasi-tree Theorem for a bouquet with exactly one non-orientable loop by delta-matroid theory and verify that the number of quasi-trees of some classes of bouquets is related to the Fibonacci and Lucas numbers by further applications of the Matrix-Quasi-tree Theorem. Furthermore, we give alternative proofs of the number of quasi-trees of these bouquets by using the deletion-contraction relations of ribbon graphs.


\section{Ribbon graphs, partial dual and partial Petrial}

\begin{definition}[\cite{bollobas}]
A {\it ribbon graph} $G$ is a $($orientable or non-orientable$)$ surface with boundary,
represented as the union of two sets of topological discs, a set $V(G)$ of vertices, and a set $E(G)$ of edges,
subject to the following restrictions.
\begin{description}
\item[(1)] The vertices and edges intersect in disjoint line segments, we call them {\it common line segments} as in \cite{Deng};
\item[(2)] Each of such common line segments lies on the boundary of precisely one vertex and precisely one edge;
\item[(3)] Every edge contains exactly two such common line segments.
\end{description}
\end{definition}


A ribbon graph $H=(V(H), E(H))$ is a \emph{ribbon subgraph} of $G=(V(G), E(G))$ if $H$ can be obtained by deleting vertices and edges of $G$. If $V(H)=V(G)$, then $H$ is a \emph{spanning ribbon subgraph} of $G$.
A \emph{quasi-tree} is a ribbon graph with exactly one boundary component. Given a ribbon graph $G$, the number of spanning ribbon subgraphs of $G$ that are quasi-trees is denoted by $\kappa(G)$. Moffatt  \cite{Moffatt15} defined the \emph{one-vertex-joint} operation on two disjoint ribbon graphs
$P$ and $Q$, denoted by $P\vee Q$, in two steps:
\begin{description}
  \item[(i)] Choose an arc $p$ on the boundary of a vertex-disc $v_1$ of $P$ that lies between two consecutive ribbon ends, and choose another such arc $q$ on the boundary of a vertex-disc $v_2$ of $Q$.
  \item[(ii)] Paste vertex-discs $v_1$ and $v_2$ together by identifying the arcs $p$ and $q$.
\end{description}
Note that $\kappa(P\vee Q)=\kappa(P)\kappa(Q)$. A ribbon graph is \emph{non-orientable} if it contains a ribbon subgraph that is homeomorphic to a M\"{o}bius band, and is \emph{orientable} otherwise. An edge $e$ of a ribbon graph is a \emph{loop} if it is incident with exactly one vertex. A loop is \emph{non-orientable} if together with its incident vertex it forms a M\"{o}bius band, and is \emph{orientable} otherwise.

For a ribbon graph $G$ and $A\subseteq E(G)$,  the \emph{partial dual} \cite{Chmutov}, $G^{\delta(A)}$,  of $G$ with respect to $A$ is the ribbon graph obtained from $G$ by gluing a disc to $G$ along each boundary component of the spanning ribbon subgraph $(V (G), A)$ (such discs will be the vertex-discs of $G^{\delta(A)}$), removing the interiors of all vertex-discs of $G$ and keeping the edge-ribbons unchanged. These include the property that $(G^{\delta(A)})^{\delta(A)}=G$.

The \emph{partial Petrial} \cite{EM},
$G^{\tau(A)}$, of $G$ with respect to $A$ is the ribbon graph obtained from $G$ by adding a half-twist to each of the edges in $A$. Similarly, $(G^{\tau(A)})^{\tau(A)}=G$. For more detailed discussions of the ribbon graphs, partial duals and partial Petrials, see \cite{EM}.

A \emph{bouquet} is a ribbon graph with exactly one vertex.
It is observed in \cite{EM} that for any connected ribbon graph $G= (V, E)$, $G^{\delta(A)}$ is a bouquet if and only if $(V, A)$ is a spanning quasi-tree of $G$. One of the fundamental relations between $G$ and $G^{\delta(A)}$ is that both have the same number of spanning quasi-trees, that is, $\kappa(G^{\delta(A)})=\kappa(G)$, see \cite{ChunJCTA} (Theorem 5.1). Hence we shall restrict ourselves to bouquets.

A \emph{signed rotation} \cite{GT} of a bouquet is a cyclic ordering of the half-edges at the vertex and if the edge is an orientable loop, then we give the same sign $+$  or $-$ to the corresponding two half-edges, and give the different signs (one $+$, the other $-$) otherwise. The sign $+$ is always omitted. Given a ribbon graph $G= (V, E)$, and $e \in E$, the \emph{deletion} of $e$ in $G$ is $G\backslash e:=(V, E\backslash e)$. The \emph{contraction} of $e$ in $G$ is $G/ e:=G^{\delta(e)}\backslash e$. Table \ref{Fig3} \cite{EM} shows the local effect of deletion and contraction on a ribbon graph.

\begin{table}
  \centering
\caption{Operations on an edge $e$ (highlighted in bold) of a ribbon graph}\label{Fig3}
  \includegraphics[width=15cm]{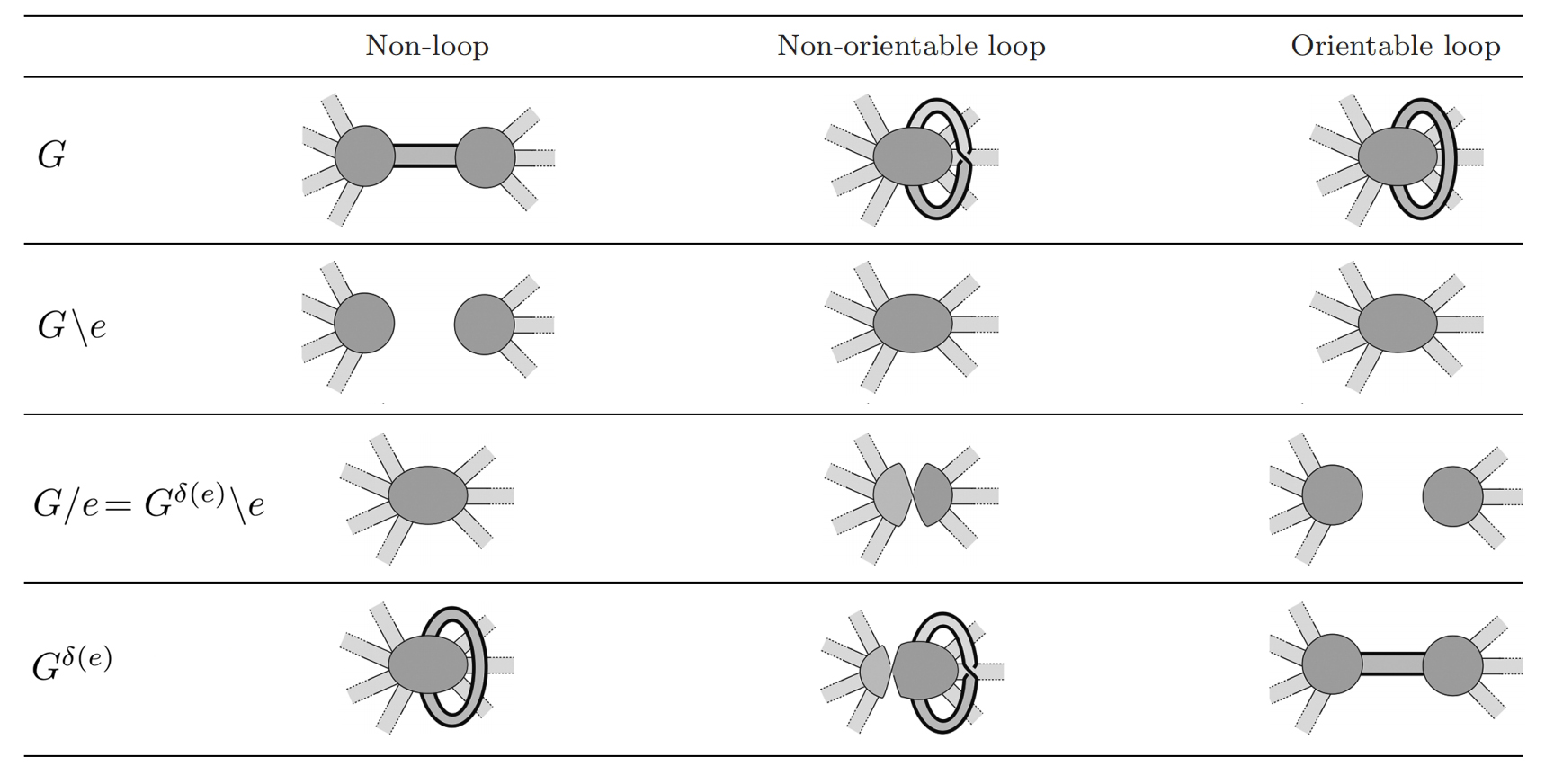}\\

\end{table}

\section{Delta-matroids and quasi-trees}

A set system $D=(E,\mathcal{F})$ is a finite set $E$ together with a subset $\mathcal{F}$ of the set $2^{E}$ of all subsets in $E$. The set $E$ is called the \emph{ground set} of the set system, and elements of $\mathcal{F}$ are its \emph{feasible sets}. $D$ is \emph{proper} if $\mathcal{F}\neq \emptyset$.

As introduced by Bouchet in \cite{AB1}, a \emph{delta-matroid} is a proper set system $D=(E, \mathcal{F})$ such that if $X, Y \in \mathcal{F}$ and $u\in X\Delta Y$, then there is $v\in X\Delta Y$ (possibly  $v=u$ ) such that $X\Delta \{u, v\}\in \mathcal{F}$.  Here $X\Delta Y:=(X\cup Y)\backslash (X\cap Y)$  is the usual symmetric difference of sets. It turns out that the edge set of a ribbon graph together with its spanning quasi-trees form a delta-matroid.

	

\begin{definition}[\cite{ChunJCTA}]
	Let $G = (V, E)$ be a ribbon graph, and let
\[\mathcal{F}:= \{F \subseteq E~|~F~\text{is the edge set of a spanning quasi-tree of} ~G\}.\]
 We call $D(G) = (E, \mathcal{F})$ the delta-matroid of $G$.
\end{definition}

Note that for any ribbon graph $G$, $\kappa(G)=|\mathcal{F}(D(G))|$. If the sizes of the feasible sets of a delta-matroid all have the same parity, then we say that the delta-matroid is \emph{even}. Otherwise, we say that the delta-matroid is \emph{odd}. For an orientable bouquet $B$, the delta-matroid $D(B)$ is even if and only if $B$ is orientable \cite{ChunJCTA}.

Let $D=(E, \mathcal{F})$ be a set system. For $A\subseteq E$, the \emph{twist} of $D$ with respect to $A$, denoted by $D*A$, is given by $$(E, \{A\Delta X~|~X\in \mathcal{F}\}).$$

Chun et al. \cite{Chun} showed that twisting and partial duality are equivalent on the delta-matroid level.

\begin{proposition}[\cite{Chun}]\label{tauandtwist}
	Let $G$ be a ribbon graph and $A\subseteq E(G)$. Then $$D(G)*A=D(G^{\delta(A)}).$$
\end{proposition}

\begin{definition}[\cite{Moffatt}]\label{def01}
	The result of handle sliding of the element $a$ over the element $b$ is the set system $\widetilde{D}_{ab}=(E, \widetilde{\mathcal{F}}_{ab})$, where $$\widetilde{\mathcal{F}}_{ab}=\mathcal{F}\triangle \{F\cup a~|~F\cup b\in \mathcal{F} ~and~ F\subseteq  E\backslash  \{a,b\}\}.$$
\end{definition}

\begin{definition}[\cite{Lando17}]\label{def02}
	The result of exchanging handle ends of the element $a$ and the element $b$ is the set system ${D'}_{ab}=(E, {\mathcal{F'}}_{ab})$, where $${\mathcal{F'}}_{ab}=\mathcal{F}\triangle \{F\cup \{a,b\}~|~F\in \mathcal{F} ~and~ F\subseteq  E\backslash  \{a,b\}\}.$$
\end{definition}

\begin{proposition}[\cite{Deng2}]\label{commute}
	The result of first exchanging handle ends and then sliding handle of the element $a$ over the element $b$ is the set system ${\widetilde{D'}}_{ab}=(E, {\mathcal{\widetilde{F'}}}_{ab})$, where $${\mathcal{\widetilde{F'}}}_{ab}=\mathcal{F}\triangle \{F\cup \{a,b\}~|~F\in \mathcal{F} ~and~ F\subseteq  E\backslash  \{a,b\}\}\triangle \{F\cup \{a\}~|~F\cup \{b\}\in \mathcal{F} ~and~ F\subseteq  E\backslash  \{a,b\}\}.$$
	Similarly, the result of first handle sliding  and then exchanging handle ends of the element $a$ and the element $b$ is the set system ${\widetilde{D}'}_{ab}={\widetilde{D'}}_{ab}$.
\end{proposition}

\begin{definition}[\cite{Lando17}]
	We say that an invariant $f$ of set systems satisfies the {\it four-term relation} if for any set system $D$ and a pair of distinct elements $a$ and $b$ in its ground set we have
	\begin{equation*}\label{4term}
		f(D)+f(\widetilde{D'}_{ab})-f(D'_{ab})-f(\widetilde{D}_{ab})=0.
	\end{equation*}
\end{definition}

\begin{theorem}\label{setsystem2}
Let $D=(E,\mathcal{F})$ be a set system and $a,b\in E$. Then the cardinality of feasible sets satisfies the four-term relation, that is,
\[|\mathcal{F}|+|\widetilde{\mathcal{F}'}_{ab}|-|\mathcal{F}{'}_{ab}|-|\widetilde{\mathcal{F}}_{ab}|=0.\]	
\end{theorem}

\begin{proof}
	For any $a,b\in E$,
let	$\mathcal{F}_{00}:=\{F~|~F\in \mathcal{F}~\text{and}~ a, b\notin F\}$,
$\mathcal{F}_{10}:=\{F~|~F\in \mathcal{F}, a\in F~\text{and}~b\notin F\}$,
$\mathcal{F}_{01}:=\{F~|~F\in \mathcal{F}, a\notin F~\text{and}~b\in F\}$ and
$\mathcal{F}_{11}:=\{F~|~F\in \mathcal{F}~\text{and}~ a, b\in F\}$. Then
$$\mathcal{F}=\mathcal{F}_{00}\cup\mathcal{F}_{10}\cup\mathcal{F}_{01}\cup\mathcal{F}_{11}.$$
By Proposition \ref{commute}, Definitions \ref{def01} and \ref{def02}, we have
\[\widetilde{\mathcal{F}}_{ab}=\mathcal{F}_{00}\cup\mathcal{F}_{01}\cup\mathcal{F}_{11}\cup(\mathcal{F}_{10}\triangle \{F\cup a~|~F\cup b\in \mathcal{F} ~and~ F\subseteq  E\backslash  \{a,b\}\}),\]
\[{\mathcal{F}}'_{ab}=\mathcal{F}_{00}\cup\mathcal{F}_{01}\cup\mathcal{F}_{10}\cup(\mathcal{F}_{11}\triangle \{F\cup \{a,b\}~|~F\in \mathcal{F} ~and~ F\subseteq  E\backslash  \{a,b\}\}),\]	
\begin{eqnarray*}
&&{\mathcal{\widetilde{F'}}}_{ab}=\mathcal{F}_{00}\cup\mathcal{F}_{01}\cup(\mathcal{F}_{11}\triangle \{F\cup \{a,b\}~|~F\in \mathcal{F} ~and~ F\subseteq  E\backslash  \{a,b\}\})\\
&&~~~~~~~~~\cup(\mathcal{F}_{10}\triangle \{F\cup \{a\}~|~F\cup \{b\}\in \mathcal{F} ~and~ F\subseteq  E\backslash  \{a,b\}\}).
\end{eqnarray*}	
It follows that
	\begin{eqnarray*}
|\mathcal{F}|+|\mathcal{\widetilde{F'}}_{ab}|&=&2|\mathcal{F}_{00}|+2|\mathcal{F}_{01}|+|\mathcal{F}_{11}|+|\mathcal{F}_{10}|+|\mathcal{F}_{11}\triangle \{F\cup \{a,b\}~|~F\in \mathcal{F} ~and~ F\subseteq  E\backslash  \{a,b\}\}|\\
	& &+|\mathcal{F}_{10}\triangle \{F\cup \{a\}~|~F\cup \{b\}\in \mathcal{F} ~and~ F\subseteq  E\backslash  \{a,b\}\}|\\
		&=&|\mathcal{\widetilde{F}}_{ab}|+|\mathcal{F}'_{ab}|.
\end{eqnarray*}
\end{proof}

\begin{proposition}[\cite{Lando17,Moffatt}]\label{slideab}
		Let $G = (V, E)$ be a ribbon graph, $a$ and $b$ be distinct edges of $G$ with neighbouring ends, let $\widetilde{G}_{ab}$ and $G'_{ab}$ be the ribbon graphs obtained from $G$ by handle sliding $a$ over $b$ and by exchange handle ends $a$ and $b$, respectively. Then
		\[D(\widetilde{G}_{ab})=\widetilde{D(G)}_{ab}.\]
		and \[D(G'_{ab})=D(G)'_{ab}.\]
\end{proposition}

In the next example, we show that the number of spanning quasi-trees of some ribbon graphs can be obtained by Theorem \ref{setsystem2} and Proposition \ref{slideab}.
\begin{example}
We have the number of quasi-trees of the  bouquet $\widetilde{(\mathbb{F}_{n})}_{1,2}~(n\geq 3)$ by the number of quasi-trees of $\mathbb{F}_{n}$, $(\mathbb{F}_{n})'_{1,2}$ and $\widetilde{(\mathbb{F}_{n})'}_{1,2}$.
Since $(\mathbb{F}_{n})'_{1,2}=\mathbb{F}_{n-1}\vee \mathbb{F}_{1}$ and   $\widetilde{(\mathbb{F}_{n})'}_{1,2}=\mathbb{F}'_{n}$, it follows that
	\begin{eqnarray*} \kappa(\widetilde{(\mathbb{F}_{n})}_{1,2})&=&\kappa({\mathbb{F}_{n}})+\kappa(\widetilde{(\mathbb{F}_{n})'}_{1,2})-\kappa((\mathbb{F}_{n})'_{1,2})\\
	&=&f_{n+1}+\ell_{n-1}-f_{n}\\
    &=&f_{n+1}+f_{n-2},
\end{eqnarray*}
where the number of quasi-trees of $\mathbb{F}_{n}$, $(\mathbb{F}_{n})'_{1,2}$ and $\widetilde{(\mathbb{F}_{n})'}_{1,2}$ are $f_{n+1}$, $f_{n}$ and $\ell_{n-1}$, respectively, see Table \ref{tab:my_label}.
\end{example}

\section{Matrix-Quasi-tree Theorem}
A \emph{chord diagram} $D$ consists of $2n$ cyclically ordered points in a core circle  together with $n$ straight line segments, called \emph{chords}, that join pairwise disjoint pairs of points $\{a_i, b_i\}$, $1 \leq i\leq n$. Two chords \emph{intersect} if the four endpoints of the chords interlace. The \emph{intersection graph} of a chord diagram $D$ is the graph $G(D)=(V,E)$ where $V$ is the set of chords, and where $uv\in E$ if and only if the chords $u$ and $v$ intersect.

A chord diagram is \emph{framed} if a map (a framing) from the set of chords to $\mathbb{Z}/ 2\mathbb{Z}$ is given, i.e., every chord is endowed with 0 or 1. 
There is a natural way to associate a framed chord diagram $F(B)$ with a bouquet $B$: take the boundary of the vertex as the circle, and place a chord between the two
ends of each edge of $B$. A chord of $F(B)$ is endowed with 0, if the corresponding ribbon loop of $B$ is orientable, and endowed with 1 otherwise.

The \emph{intersection matrix} $A(D)$ of a framed chord diagram $D =\{\{a_{i},b_{i}\}~|~ 1
\leq i\leq n\}$ is constructed as follows.
First, choose an arbitrary ordered pair $(a_{i}, b_{i})$ or $(b_{i},a_{i})$ for each chord $\{a_{i},b_{i}\}$. The entry $A_{i,i}$ is 1 if the corresponding chord with framing 1, and entry $A_{i,i}$ is 0 otherwise. For $i < j$, entry $A_{i,j}$ is 1 if the chords intersect and the endpoints in the corresponding ordered pairs are in cyclic order $a_{i}, a_{j}, b_{i}, b_{j}$, and $-1$ if the cyclic order is $a_{i}, b_{j}, b_{i}, a_{j}$, and entry $A_{i,j}$ is 0 otherwise. For $j < i$, $A_{i,j}=-A_{j,i}$.

More important for us is the following Matrix-Quasi-tree Theorem for orientable bouquets
obtained in many different contexts by different authors \cite{Bouchet87,Lauri,Macris,Merino2023}.

\begin{theorem}[\cite{Bouchet87,Lauri,Macris,Merino2023}]$($Matrix-Quasi-tree Theorem$)$\label{key1}
	Given an orientable bouquet $B$ with $n$ edges,  $$\kappa(B)=\det(I_{n} + A(F({B}))).$$
\end{theorem}

The fact that $\det(I_{n} + A(F({B})))$ does not change if $a_i$ and $b_i$ are interchanged in the ordered pair $(a_i, b_i)$ was proved in \cite{Bouchet87, Naji}. Throughout this paper, we will often write $A(B)$ for $A(F(B))$ when the framed chord diagram of $B$ is clear from the context.

Let $D=(E, \mathcal{F})$ be a set system and $e\in E$. Denote the set of the elements of $\mathcal{F}$ containing $e$ by $\mathcal{F}_{e}$ and let $\mathcal{F}_{\overline{e}}:=\mathcal{F}\backslash \mathcal{F}_{e}$.
We define \emph{loop complementation} of $D$ on $e$, denoted by $D+e$, as $(E, \mathcal{F''})$,
where $$\mathcal{F''}=\mathcal{F}
\triangle \{F \cup \{e\}~|~F \in \mathcal{F}_{\overline{e}}\}.$$
If $D$ is an even delta-matroid, then for any $F_{1}\in \mathcal{F}$ and $F_{2}\in  \{F\cup e| F\in \mathcal{F}_{\overline{e}}\}$, $|F_{1}|$ and $|F_{2}|$ have different parity. Therefore,
$\mathcal{F''}=\mathcal{F}\cup \{F \cup \{e\}~|~F \in \mathcal{F}_{\overline{e}}\}.$ Hence,
\begin{equation}\label{equ1}
  |\mathcal{F}''|=|\mathcal{F}|+|\mathcal{F}_{\overline{e}}|.
\end{equation}

\begin{proposition}[\cite{Chun}]\label{tauandloopcom}
	Let $G$ be a ribbon graph and $e\in E(G)$. Then $$D(G)+e=D(G^{\tau(e)}).$$
\end{proposition}



Let $M$ be a square matrix over $\mathbb{R}$ with rows and columns indexed, in the same order, by some set $E$. For $X\subseteq E$, we let $M[X]$ submatrix given by the rows and columns indexed by $X$. We denote the submatrix of $M$ obtained by deleting row $i$ and column $j$ by $M[i, j]$. Recall that a \emph{unimodular matrix} is a matrix of determinant 1 or $-1$. A \emph{principal unimodular matrix} is such that every nonsingular principal submatrix is unimodular.

\begin{theorem}\label{quasitreeTH}$($Matrix-Quasi-tree Theorem$)$
Let $B$ be a bouquet with $n$ edges which contains exactly one non-orientable loop $e_{1}$. Then the number of quasi-trees of $B$ equals $\det(I_{n}+A(B))$.
\end{theorem}

\begin{proof}

Note that $B^{\tau(e_{1})}$ is an orientable bouquet. Then $D(B^{\tau(e_{1})})=(E, \mathcal{F})$ is an even delta-matroid. By Proposition \ref{tauandloopcom},
$$D(B)=D((B^{\tau(e_{1})})^{\tau(e_{1})})=D(B^{\tau(e_{1})})+e_{1}.$$ Then by equation (\ref{equ1}), we have $$|\mathcal{F}(D(B))|=|\mathcal{F}''|=|\mathcal{F}|+|\mathcal{F}_{\overline{e_{1}}}|.$$	
Let $M:=I_{n}+A(B)$. First, we expand the determinant of $M$ along the first column. The submatrix $M[1, 1]$ is the matrix $I_{n-1}+A(B\backslash e_{1})$ whose determinant equals $\kappa(B\backslash e_{1})$ by Theorem \ref{key1}, that is, $$\det(M[1, 1])=\det(I_{n-1}+A(B\backslash e_{1}))=|\mathcal{F}_{\overline{e_{1}}}|.$$ Moreover,
\begin{eqnarray*}
		& &\det(I_{n}+A(B))= 2\cdot \det(M[1, 1])+\sum_{i=2}^{n}(-1)^{1+i}m_{i,1}\cdot \det(M[i,1])\\
		&=&  \det(I_{n-1}+A(B\backslash e_{1}))+\det(M[1, 1])+\sum_{i=2}^{n}(-1)^{1+i}m_{i,1}\cdot \det(M[i,1])\\
		&=&  \det(I_{n-1}+A({B}\backslash e_{1}))+\det((I_{n}+A({B^{\tau(e_{1})}}))[1, 1])+\sum_{i=2}^{n}(-1)^{1+i}m_{i,1}\cdot \det((I_{n}+A({B^{\tau(e_{1})}}))[i,1])\\	
		&=&  \det(I_{n-1}+A({B}\backslash e_{1}))+\det(I_{n}+A({B^{\tau(e_{1})}}))\\
		&=&|\mathcal{F}_{\overline{e_{1}}}|+|\mathcal{F}|,
\end{eqnarray*}
where the last equality follows from Theorem \ref{key1}.
Hence, $$\kappa(B)=|\mathcal{F}(D(B))|=\det(I_{n}+A(B)).$$
\end{proof}

\begin{remark}
Let $B$ be a bouquet with $n$ edges which contains exactly one non-orientable loop $e_{1}$. By Theorem \ref{quasitreeTH}, $\det(I_{n}+A(B))$ does not change if $a_i$ and $b_i$ are interchanged in the ordered pair $(a_i, b_i)$ in the process of obtaining intersection matrix $A(B)$.
\end{remark}

Merino et al. \cite{Merino2023} obtained the following statement by using directed connected maps and principally unimodular matrices.
\begin{proposition}[\cite{Merino2023}]\label{orient}
	Let $B$ be an orientable bouquet with $n$ edges and $X\subseteq E(B)$. Then
	\begin{eqnarray*}
		\det(A(B)[X])=\left\{\begin{array}{ll}
			1, & \mbox{if}~X~\mbox{is the edge set of a quasi-tree of}~B,\\
			0, & \mbox{otherwise}.
		\end{array}\right.
	\end{eqnarray*}
\end{proposition}


\begin{proposition}\label{nonorient-1}
Let $B$ be a bouquet with $n$ edges which contains exactly one non-orientable loop $e_{1}$, and let $X\subseteq E(B)$. Then
	\begin{eqnarray*}
		\det(A(B)[X])=\left\{\begin{array}{ll}
			1, & \mbox{if}~X~\mbox{is the edge set of a quasi-tree of}~B,\\
			0, & \mbox{otherwise}.
		\end{array}\right.
	\end{eqnarray*}
\end{proposition}

\begin{proof}
Clearly the result is true when $e_{1}\notin X$ by Proposition \ref{orient}. If $e_{1}\in X$, we have
$$\det(A(B)[X])= \det(A(B^{\tau(e_{1})})[X\backslash e_{1}])+\det(A(B^{\tau(e_{1})})[X]).$$
Recall that $D(B^{\tau(e_{1})})=(E,\mathcal{F})$ is an even delta-matroid and the cardinality of each feasible set is even. Since $\mathcal{F}=\mathcal{F}_{e_{1}}\cup \mathcal{F}_{\overline{e_{1}}}$, it follows that $$\mathcal{F}(D(B))=\mathcal{F}_{e_{1}}\cup \mathcal{F}_{\overline{e_{1}}}\cup \{F \cup e_{1}~|~F \in \mathcal{F}_{\overline{e_{1}}}\}.$$
Note that for any $F_1\in \mathcal{F}_{e_{1}}\cup \mathcal{F}_{\overline{e_{1}}}$ and $F_2\in \{F \cup e_{1} | F \in \mathcal{F}_{\overline{e_{1}}}\} $, $|F_1|$ and $|F_2|$ are even and odd, respectively. Moreover,  $$\mathcal{F}_{e_{1}}\cap \mathcal{F}_{\overline{e_{1}}}\cap \{F \cup e_{1} | F \in \mathcal{F}_{\overline{e_{1}}}\}=\emptyset.$$
We now distinguish two cases:

\begin{description}
  \item[Case 1.] If $X$ is even, then $X\backslash e_{1}\notin \mathcal{F}$. Thus $\det(A(B^{\tau(e_{1})})[X\backslash e_{1}])=0$ by Proposition \ref{orient}.
\begin{description}
  \item[Case 1.1] If $X$ is the edge set of a quasi-tree of $B$, that is, $X\in\mathcal{F}(D(B))$, then  $X\in \mathcal{F}_{e_{1}}$. Thus $X$ is a quasi-tree of $B^{\tau(e_{1})}$. Therefore $\det(A(B^{\tau(e_{1})})[X])=1$ by Proposition \ref{orient}.
	Hence $$\det(A(B)[X])= \det(A(B^{\tau(e_{1})})[X\backslash e_{1}])+\det(A(B^{\tau(e_{1})})[X])=1.$$
  \item[Case 1.2] If $X$ is not the edge set of a quasi-tree of $B$, that is, $X\notin\mathcal{F}(D(B))$, then $X\notin \mathcal{F}_{e_{1}}$. Therefore $X$ is not a quasi-tree of $B^{\tau(e_{1})}$. Thus $\det(A(B^{\tau(e_{1})})[X])=0$ by Proposition \ref{orient}.
	Hence $$\det(A(B)[X])= \det(A(B^{\tau(e_{1})})[X\backslash e_{1}])+\det(A(B^{\tau(e_{1})})[X])=0.$$
\end{description}
  \item[Case 2.] If $X$ is odd, then $X\notin \mathcal{F}$. We have $\det(A(B^{\tau(e_{1})})[X])=0$ by Proposition \ref{orient}.
\begin{description}
  \item[Case 2.1] If $X$ is the edge set of a quasi-tree of $B$, then $X\in \{F\cup e_{1}| F\in \mathcal{F}_{\overline{e_{1}}}\}$. Thus $X\backslash e_{1}\in \mathcal{F}_{\overline{e_{1}}}$, that is, $X\backslash e_{1}$ is a quasi-tree of $B^{\tau(e_{1})}$. Therefore $\det(A(B^{\tau(e_{1})})[X\backslash e_{1}])=1$ by Proposition \ref{orient}. Hence $$\det(A(B)[X])= \det(A(B^{\tau(e_{1})})[X\backslash e_{1}])+\det(A(B^{\tau(e_{1})})[X])=1.$$
  \item[Case 2.2] If $X$ is not the edge set of a quasi-tree of $B$, then $X\notin \{F\cup e_{1}| F\in \mathcal{F}_{\overline{e_{1}}}\}$. We have $X\backslash e_{1}\notin \mathcal{F}_{\overline{e_{1}}}$, that is, $X\backslash e_{1}$ is not a quasi-tree of $B^{\tau(e_{1})}$. Hence $\det(A(B^{\tau(e_{1})})[X\backslash e_{1}])=0$ by Proposition \ref{orient}. Therefore $$\det(A(B)[X])= \det(A(B^{\tau(e_{1})})[X\backslash e_{1}])+\det(A(B^{\tau(e_{1})})[X])=0.$$
\end{description}
\end{description}
\end{proof}

The following corollary is an immediate consequence of Proposition \ref{nonorient-1}.

\begin{corollary}\label{principleuni}
	Let $B$ be a bouquet with $n$ edges which contains exactly one non-orientable loop $e_{1}$. Then $A(B)$ is a principal unimodular matrix.
\end{corollary}

\begin{remark}
It is well-known that
\[\det(I+M)=\sum_{X\subseteq E}\det(M[X]),\]
where $M$ is a square matrix, the rows and columns of $M$ are indexed by $E$, and $\det(M[\emptyset])=1$. We can obtain the Matrix Quasi-tree Theorem $($Theorem \ref{quasitreeTH}$)$ also by Proposition \ref{nonorient-1}.
\end{remark}

\section{The Fibonacci numbers, Lucas numbers and some classes of bouquets}

\subsection{The number of quasi-trees of $\mathbb{F}'_n, \mathbb{F}'^{1}_{n}, \mathbb{F}^{1}_{n}, \mathbb{F}'^{n}_{n}$ and $\mathbb{W}^{i}_{n}$ }
For each $n\geq2$, let $\mathbb{F}'_n$ be the bouquet with the signed rotation
$$(1,2,3,2,1,4,3,5,4,\cdots,i, i-1,i+1,i,\cdots,n-1,n-2,n,n-1,n),$$ whose chord diagram $D(\mathbb{F}'_n)$ consists of the pairs $$\{(1,5),(2,4),(3,7),(6,9),\cdots,(2n-4,2n-1),(2n-2,2n)\}.$$
Figure \ref{fig:newfn} shows the bouquet $\mathbb{F}'_6$. The corresponding intersection graph of $\mathbb{F}'_n$ is the caterpillar $T_n$. Figure \ref{fig:caterpillar} shows  $T_7$.

\begin{figure}
	\centering
	\includegraphics[width=0.5\linewidth]{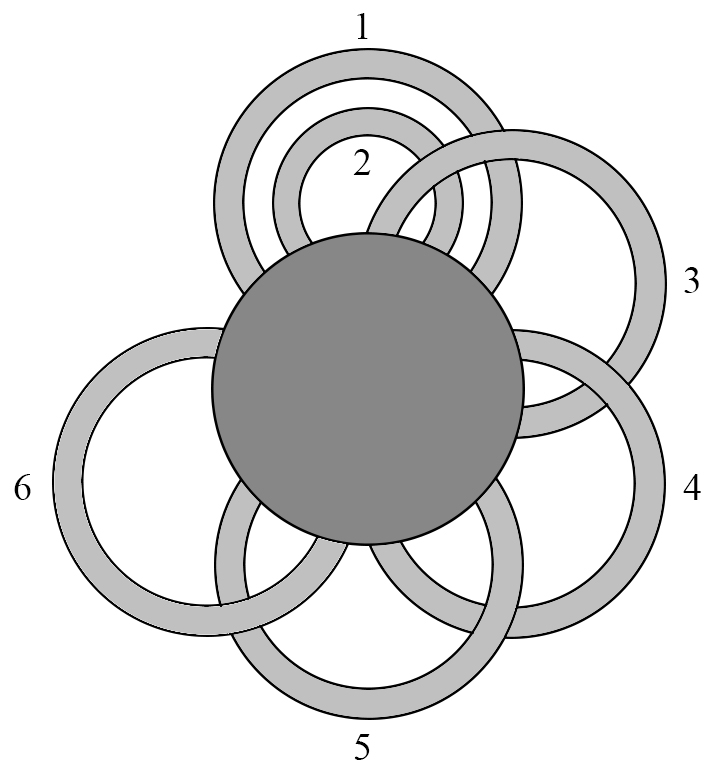}
	\caption{$\mathbb{F}'_6$}
	\label{fig:newfn}
\end{figure}

\begin{theorem}\label{main1}
	The number of quasi-trees of $\mathbb{F}'_{n}~(n\geq 2)$ equals the $(n-1)$-th Lucas number $\ell_{n-1}$.
\end{theorem}
\begin{proof}
	If $n=2$, then $\kappa(\mathbb{F}'_{2})=1=\ell_{1}$. If $n\geq 3$, then

\[I_{n}+A(\mathbb{F}'_n)=\left(
\begin{array}{ccccccc}
	1 & 0 & 1 &\cdots& 0 & 0 & 0 \\
	0 & 1 & 1 & 0 & \cdots & 0 & 0 \\
	-1&-1 & 1 & 1 & \cdots & 0 & 0 \\
	\vdots&\vdots&\vdots&\vdots&\vdots&\vdots& \vdots\\
	0 & 0 & 0 & \cdots & 1 & 1 & 0 \\
	0 & 0 & 0 & \cdots &-1 & 1 & 1 \\
	0 & 0 & 0 & \cdots & 0 &-1 & 1 \\
\end{array}
\right).\]
The number of quasi-trees of $\mathbb{F}'_n$ is given by the determinant of the matrix $I_{n}+A(\mathbb{F}'_n)$ by Theorem \ref{key1}. Let $M:=I_{n}+A(\mathbb{F}'_n)$. Then
\begin{eqnarray*}
	\kappa(\mathbb{F}'_n)&=&\det(I_n+A(\mathbb{F}'_n))\\
	&=&  \det(M[1, 1])+\sum_{i=2}^{n}(-1)^{1+i}m_{i,1}\cdot \det(M[i,1])\\
	&=& \kappa(\mathbb{F}_{n-1})+(-1)^{3}m_{1,3}m_{2,2}m_{3,1}\kappa(\mathbb{F}_{n-3})\\
	&=&\kappa(\mathbb{F}_{n-1})+\kappa(\mathbb{F}_{n-3})\\
	&=&f_{n}+f_{n-2}\\
	&=&\ell_{n-1}.
\end{eqnarray*}
\end{proof}

Let $\mathbb{F}'^{1}_{n}$ denote the non-orientable bouquet obtained by twisting the first edge in $\mathbb{F}'_{n}$.
Then the signed rotation of  $\mathbb{F}'^{1}_{n}$ is
 $$(-1,2,3,2,1,4,3,5,4,\cdots,i, i-1,i+1,i,\cdots,n-1,n-2,n,n-1,n),$$ whose framed chord diagram
  $D(\mathbb{F}'^{1}_{n})$ consists of the pairs $$\{(1,5),(2,4),(3,7),(6,9),\cdots,(2n-4,2n-1),(2n-2,2n)\},$$ where chord $(1,5)$ is endowed with 1, others are endowed with 0.

\begin{theorem}
The number of quasi-trees of~~$\mathbb{F}'^{1}_{n}~(n\geq 2)$ is $f_{n}+\ell_{n-1}$.
\end{theorem}

\begin{proof}
	
If $n=2$, then $\kappa(\mathbb{F}'^{1}_{2})=2=f_{2}+\ell_{1}$. For $n\geq 3$,
\[I_{n}+A(\mathbb{F}'^{1}_{n})=\left(
\begin{array}{ccccccc}
	2 & 0 & 1 &\cdots& 0 & 0 & 0 \\
	0 & 1 & 1 & 0 & \cdots & 0 & 0 \\
	-1&-1 & 1 & 1 & \cdots & 0 & 0 \\
\vdots&\vdots&\vdots&\vdots&\vdots&\vdots& \vdots\\
	0 & 0 & 0 & \cdots & 1 & 1 & 0 \\
	0 & 0 & 0 & \cdots &-1 & 1 & 1 \\
	0 & 0 & 0 & \cdots & 0 &-1 & 1 \\
\end{array}
\right).\]
By Theorem \ref{quasitreeTH}, the number of quasi-trees of $\mathbb{F}'^{1}_{n}$ is given by the determinant of the matrix $I_{n}+A(\mathbb{F}'^{1}_{n})$. Let $M:=I_{n}+A(\mathbb{F}'^{1}_{n})$.
First, we expand the determinant of $M$ along the first column. The submatrix $M[1, 1]$ is the matrix $I_{n-1}+A(\mathbb{F}'^{1}_{n}\backslash e_{1})$ whose determinant equals the number of quasi-trees of $\mathbb{F}'^{1}_{n}\backslash e_{1}$ by Theorem \ref{key1}. Note that $\mathbb{F}'^{1}_{n}\backslash e_{1}=\mathbb{F}_{n-1}$. Hence,
\begin{eqnarray*}
	\det(I_{n}+A(\mathbb{F}'^{1}_{n}))&=& 2\cdot \det(M[1, 1])+\sum_{i=2}^{n}(-1)^{1+i}m_{i,1}\cdot \det(M[i,1])\\
	&=&  2 \kappa(\mathbb{F}_{n-1})+(-1)^{3}m_{1,3}m_{2,2}m_{3,1}\kappa(\mathbb{F}_{n-3})\\
		&=&2 \kappa(\mathbb{F}_{n-1})+\kappa(\mathbb{F}_{n-3})\\
		&=&2f_{n}+f_{n-2}\\
		&=&f_{n}+\ell_{n-1}.
\end{eqnarray*}
\end{proof}

Let $\mathbb{F}^{1}_{n}$ denote the non-orientable bouquet obtained by twisting the first edge in $\mathbb{F}_n$.
Then the signed rotation of $\mathbb{F}^{1}_{n}$ is $$(-1,2,1,3,2,4,3,\cdots,i, i-1,i+1,i,\cdots,n-1,n-2,n,n-1,n),$$ whose framed chord diagram $D(\mathbb{F}^{1}_{n})$ consists of the pairs $$\{(1,3),(2,5),(4,7),\cdots,(2n-4,2n-1),(2n-2,2n)\},$$ where chord $(1,3)$  is endowed with 1, others are endowed with 0.

\begin{theorem}
The number of quasi-trees of $\mathbb{F}^{1}_{n}~(n\geq 1)$ is $(n+2)$-th Fibonacci number $f_{n+2}$.
\end{theorem}

\begin{proof}
If $n=1$, then $\kappa(\mathbb{F}^{1}_{1})=2=f_{3}$.  For $n\geq 2$,
\[I_{n}+A(\mathbb{F}^{1}_{n})=\left(
\begin{array}{ccccccc}
	2 & 1 & 0 & \cdots & 0 &  0 & 0 \\
	-1 & 1 & 1 & 0 & \cdots & 0 & 0 \\
	0 & -1 & 1 & 1 & \cdots & 0 & 0 \\
	\vdots & \vdots & \vdots & \vdots & \vdots & \vdots& \vdots  \\
	0 & 0 & 0 & \cdots & 1 & 1 & 0 \\
	0 & 0 & 0 & \cdots & -1 & 1 & 1 \\
	0 & 0& 0 & \cdots & 0 & -1  & 1 \\
\end{array}
\right).\]
By Theorem \ref{quasitreeTH},
the number of quasi-trees of $\mathbb{F}^{1}_{n}$ is given by the determinant of the matrix $I_{n}+A(\mathbb{F}^{1}_{n})$. Let $M:=I_{n}+A(\mathbb{F}^{1}_{n})$.
First, we expand the determinant of $M$ along the first column. The submatrix $M[1, 1]$ is the matrix $I_{n-1}+A(\mathbb{F}^{1}_{n}\backslash e_{1})$ whose determinant equals the number of quasi-trees of $\mathbb{F}^{1}_{n}\backslash e_{1}$ by Theorem \ref{key1}. Note that
$\mathbb{F}^{1}_{n}\backslash e_{1}=\mathbb{F}_{n-1}$. Therefore,
	\begin{eqnarray*}
		\det(I_{n}+A(\mathbb{F}^{1}_{n}))&=& 2\cdot \det(M[1, 1])+\sum_{i=2}^{n}(-1)^{1+i}m_{i,1}\cdot \det(M[i,1])\\
		&=&  2 \kappa(\mathbb{F}_{n-1})+(-1)^{1}m_{1,2}m_{2,1}\kappa(\mathbb{F}_{n-2})\\
		&=&2 \kappa(\mathbb{F}_{n-1})+\kappa(\mathbb{F}_{n-2})\\
		&=&2f_{n}+f_{n-1}\\
		&=&f_{n+2}.
	\end{eqnarray*}
\end{proof}

Let $\mathbb{F}'^{n}_{n}$ denote the non-orientable bouquet obtained by twisting the last edge in $\mathbb{F}'_{n}$. Then  the signed rotation of $\mathbb{F}'^{n}_{n}$ is $$(1,2,3,2,1,4,3,5,4,\cdots,i, i-1,i+1,i,\cdots,n-1,n-2,n,n-1,-n),$$ whose framed chord diagram $D(\mathbb{F}'^{n}_{n})$ consists of the pairs $$\{(1,5),(2,4),(3,7),(6,9),\cdots,(2n-4,2n-1),(2n-2,2n)\},$$ where chord $(2n-2,2n)$ is endowed with 1, others are endowed with 0.

\begin{theorem}
The number of quasi-trees of $\mathbb{F}'^{n}_{n}~(n\geq 3)$ is $n$-th Lucas number $\ell_{n}$.
\end{theorem}

\begin{proof}
If $n=3$, then $\kappa(\mathbb{F}'^{3}_{3})=4=\ell_{3}$. For $n\geq 4$,
\[I_{n}+A(\mathbb{F}'^{n}_{n})=\left(
\begin{array}{ccccccc}
	1 & 0 & 1 &\cdots& 0 & 0 & 0 \\
	0 & 1 & 1 & 0 & \cdots & 0 & 0 \\
	-1&-1 & 1 & 1 & \cdots & 0 & 0 \\
	\vdots&\vdots&\vdots&\vdots&\vdots&\vdots& \vdots\\
	0 & 0 & 0 & \cdots & 1 & 1 & 0 \\
	0 & 0 & 0 & \cdots &-1 & 1 & 1 \\
	0 & 0 & 0 & \cdots & 0 &-1 & 2 \\
\end{array}
\right).\]
By Theorem \ref{quasitreeTH}, the number of quasi-trees of $\mathbb{F}'^{n}_{n}$ is given by the determinant of the matrix $I_{n}+A(\mathbb{F}'^{n}_{n})$.
Let $M:=I_{n}+A(\mathbb{F}'^{n}_{n}).$
We expand the determinant of $M$ along the first column. The submatrix $M[1, 1]$ is the matrix $I_{n-1}+A(\mathbb{F}'^{n}_{n}\backslash e_{1})$ whose determinant equals the number of quasi-trees of $\mathbb{F}'^{n}_{n}\backslash e_{1}$ by Theorem \ref{quasitreeTH}. Note that $\mathbb{F}'^{n}_{n}\backslash e_{1}=\mathbb{F}^{1}_{n-1}.$ Hence,
	\begin{eqnarray*}
		\det(I_{n}+A(\mathbb{F}'^{n}_{n}))&=& \det(M[1, 1])+\sum_{i=2}^{n}(-1)^{1+i}m_{i,1}\cdot \det(M[i,1])\\
		&=&   \kappa(\mathbb{F}^{1}_{n-1})+(-1)^{3}m_{1,3}m_{2,2}m_{3,1}\kappa(\mathbb{F}^{1}_{n-3})\\
		&=& \kappa(\mathbb{F}^{1}_{n-1})+\kappa(\mathbb{F}^{1}_{n-3})\\
		&=&f_{n+1}+f_{n-1}\\
		&=&\ell_{n}.
	\end{eqnarray*}
\end{proof}


Let $\mathbb{W}^{i}_{n}~(1\leq i\leq n)$ denote the non-orientable bouquet obtained by twisting the $i$-th edge in $\mathbb{W}_n$. Note that $\mathbb{W}^{i}_{n}=\mathbb{W}^{1}_{n}$.
Then we discuss the number of quasi-trees of $\mathbb{W}^{1}_{n}$.
The signed rotation of $\mathbb{W}^{1}_{n}$ is $$(-1,n,2,1,3,2,4,3,\cdots,i, i-1,i+1,i,\cdots,n-1,n-2,n,n-1),$$ whose framed chord diagram $D(\mathbb{W}^{1}_{n})$ consists of the pairs $$\{(1,4),(3,6),(5,8),\cdots,(2n-3,2n),(2n-1,2)\},$$ where chord $(1,4)$ is endowed with 1, others are endowed with 0.

\begin{theorem}
The number of quasi-trees of $\mathbb{W}^{1}_{n}~(n\geq 3)$ is $2f_{n+1}-1+(-1)^{n+1}$.
\end{theorem}

\begin{proof}
For $n\geq 3$,
\[I_{n}+A(\mathbb{W}^{1}_{n})=\left(
\begin{array}{ccccccc}
	2 & 1 & 0 & \cdots & 0 &  0 & -1 \\
	-1 & 1 & 1 & 0 & \cdots & 0 & 0 \\
	0 & -1 & 1 & 1 & \cdots & 0 & 0 \\
	\vdots & \vdots & \vdots & \vdots & \vdots & \vdots& \vdots  \\
	0 & 0 & 0 & \cdots & 1 & 1 & 0 \\
	0 & 0 & 0 & \cdots & -1 & 1 & 1 \\
1 & 0& 0 & \cdots & 0 & -1  & 1 \\
\end{array}
\right).\]
By Theorem \ref{quasitreeTH}, the number of quasi-trees of $\mathbb{W}^{1}_{n}$ is given by the determinant of the matrix $I_{n}+A(\mathbb{W}^{1}_{n})$.
Let $M:=I_{n}+A(\mathbb{W}^{1}_{n})$. We expand the determinant of $M$ along the first column. The submatrix $M[1, 1]$ is the matrix $I_{n-1}+A(\mathbb{W}^{1}_{n}\backslash e_{1})$ whose determinant equals the number of quasi-trees of $\mathbb{W}^{1}_{n}\backslash e_{1}$ by Theorem \ref{key1}. Note that $\mathbb{W}^{1}_{n}\backslash e_{1}=\mathbb{F}_{n-1}$. Hence
	\begin{eqnarray*}
		\det(I_{n}+A(\mathbb{W}^{1}_{n}))&=&\sum_{i=1}^{n}(-1)^{i+1}m_{i,1}\cdot \det(M[i,1])\\
		&=& 2\cdot \det(M[1, 1])+(-1)^{2+1}m_{2,1}\cdot \det(M[2,1])+(-1)^{n+1}m_{n,1}\cdot \det(M[n,1])\\
		&=&2\kappa(\mathbb{F}_{n-1})+ \det(M[2,1])+(-1)^{n+1}\cdot \det(M[n,1]).
	\end{eqnarray*}
Since
\begin{eqnarray*}
\det(M[2,1])&=& 1\cdot \det(M[2,1][1,1])+(-1)^{1+n-1}\cdot(-1)\cdot \det(M[2,1][1,n-1])\\
 &=&\kappa(\mathbb{F}_{n-2})+(-1)^{n+1}(-1)^{n-2}\\
 &=&\kappa(\mathbb{F}_{n-2})-1\\
  &=&f_{n-1}-1,
\end{eqnarray*}
and
\begin{eqnarray*}
	\det(M[n,1]) &=& 1\cdot \det(M[n,1][1,1])+(-1)^{1+n-1}\cdot(-1)\cdot \det(M[n,1][1,n-1])\\
	&=&1+(-1)^{n+1}\kappa(\mathbb{F}_{n-2})\\
	&=&1+(-1)^{n+1}f_{n-1}.
\end{eqnarray*}
It follows that
\begin{eqnarray*}
\kappa(\mathbb{W}^{1}_{n})&=&2\kappa(\mathbb{F}_{n-1})+ \det(M[2,1])+(-1)^{n+1}\cdot \det(M[n,1])\\	
&=&2\kappa(\mathbb{F}_{n-1})+(f_{n-1}-1)+(-1)^{n+1}(1+(-1)^{n+1}f_{n-1})\\
		&=&2f_{n}+f_{n-1}-1+(-1)^{n+1}+f_{n-1}\\
			&=&2f_{n+1}-1+(-1)^{n+1}.
	\end{eqnarray*}
\end{proof}

We list the number of quasi-trees of bouquets  (discussed in Section 5.1) as shown in Table \ref{tab:my_label}.

 \begin{table}
	\centering
	\caption{The number of quasi-trees of some classes of bouquets}
	\begin{tabular}{ccccc}
		\hline
		~~~~~~Bouquet $B$~~~~~~& ~~~~~~$\kappa(B)$~~~~~~ \\
		\hline
		$\mathbb{F}_{n},~~n\geq 0$	& ${f}_{n+1}$  \\
	\hline$\mathbb{W}_{n},~~n\geq 3$	& ${a}_{n}$  \\
	\hline$\mathbb{F}'_{n},~~n\geq 2$	& ${\ell}_{n-1}$   \\
	\hline$\mathbb{F}^{1}_{n},~~n\geq 1$	& ${f}_{n+2}$ \\		
	\hline$\mathbb{W}^{i}_{n},~~n\geq 3$ 	&  $2f_{n+1}-1+(-1)^{n+1}$ \\
		\hline$\mathbb{F}'^{1}_{n},~~n\geq 2$	& ${f}_{n}+\ell_{n-1}$   \\
		\hline$\mathbb{F}'^{n}_{n},~~n\geq 3$	& $\ell_{n}$   \\
\hline
	\end{tabular}
	
	\label{tab:my_label}
\end{table}

\subsection{Ribbon graph theory proof}

Let $G=(V, E)$ be a ribbon graph and let $n$ be positive integer.
Then we define
$$\mathcal{F}_n(G)=\{F\subseteq E~|~\text{the number of boundary components of}~ (V, F)~\text{is}~n\},$$
and its cardinality will be denoted by $f_n(G)$.

\begin{theorem}\label{key5}
	Let G be a ribbon graph and $e\in E(G)$. Then
	
	\[f_n(G)=f_n(G\backslash e)+ f_n(G/e).\]
In particular $($when $n=1)$ \[\kappa(G)=\kappa(G\backslash e)+ \kappa(G/e).\]
\end{theorem}

\begin{proof}
Note that $F\in \mathcal{F}_n(G)$ and  $e\notin F$ if and only if $F\in \mathcal{F}_n(G\backslash e)$. From Table \ref{Fig3} we see
that $F\in \mathcal{F}_n(G)$ and $e\in F$ if and only if $F\backslash e\in \mathcal{F}_n(G^{\delta(e)})$ if and only if $F\backslash e\in \mathcal{F}_n(G^{\delta(e)}\backslash e)=\mathcal{F}_n(G/e)$.
Hence there is a natural 1-1 correspondence between the element of $\mathcal{F}_n(G)$ and the element of $\mathcal{F}_n(G\backslash e)$ or $\mathcal{F}_n(G/e)$. It follows that
\[f_n(G)=f_n(G\backslash e)+ f_n(G/e).\]
\end{proof}

\begin{corollary}
Let $G$ be a ribbon graph and $A\subseteq E(G)$. Then $$f_n(G)=f_n(G^{\delta(A)}).$$
In particular $($when $n=1)$ \[\kappa(G)=\kappa(G^{\delta(A)}).\]
\end{corollary}

\begin{proof}
It suffices to prove this for a single edge $e$ of $A$, i.e., $f_n(G)=f_n(G^{\delta(e)})$.
Since $G/e=G^{\delta(e)}\backslash e$ and $G\backslash e=(G^{\delta(e)})^{\delta(e)}\backslash e=G^{\delta(e)}/e$, it follows that
$$f_n(G)=f_n(G\backslash e)+f_n(G/e)=f_n(G^{\delta(e)}/e)+f_n(G^{\delta(e)}\backslash e)
=f_n(G^{\delta(e)}),$$
by Theorem \ref{key5}.
\end{proof}


We give alternative proofs of the number of quasi-trees of bouquets as shown in Table \ref{tab:my_label} by Theorem \ref{key5}.

\begin{description}
  \item[Alternative proof of $\kappa(\mathbb{F}'_n)~(n\geq 2)$.]  \[\kappa(\mathbb{F}'_n)=\kappa(\mathbb{F}'_n\backslash e_n)+ \kappa(\mathbb{F}'_n/ e_n)=\kappa(\mathbb{F}'_{n-1})+ \kappa(\mathbb{F}'_{n-2}\vee P_2),\]
and
\[\kappa(\mathbb{F}'_{n-2}\vee P_2)=\kappa((\mathbb{F}'_{n-2}\vee P_2)\backslash e_{n-1})+\kappa( (\mathbb{F}'_{n-2}\vee P_2) / e_{n-1})=0+\kappa(\mathbb{F}'_{n-2})=\kappa(\mathbb{F}'_{n-2}).\]
Therefore,
\[\kappa(\mathbb{F}'_n)=\kappa(\mathbb{F}'_{n-1})+ \kappa(\mathbb{F}'_{n-2}).\]
By calculation, $\kappa(\mathbb{F}'_2)=1=\ell_{1}$ and  $\kappa(\mathbb{F}'_3)=3=\ell_{2}$, we obtain that $\kappa(\mathbb{F}'_n)=\ell_{n-1}.$
  \item[Alternative proof of $\kappa(\mathbb{F}^{1}_n)~(n\geq 1)$.]  \[\kappa(\mathbb{F}^{1}_n)=\kappa(\mathbb{F}^{1}_n\backslash e_n)+ \kappa(\mathbb{F}^{1}_n/ e_n)=\kappa(\mathbb{F}^{1}_{n-1})+ \kappa(\mathbb{F}^{1}_{n-2}\vee P_2),\]
and
\[\kappa(\mathbb{F}^{1}_{n-2}\vee P_2)=\kappa((\mathbb{F}^{1}_{n-2}\vee P_2) \backslash e_{n-1})+\kappa( (\mathbb{F}^{1}_{n-2}\vee P_2) / e_{n-1})=0+\kappa(\mathbb{F}^{1}_{n-2})=\kappa(\mathbb{F}^{1}_{n-2}).\]
It follows that
\[\kappa(\mathbb{F}^{1}_n)=\kappa(\mathbb{F}^{1}_{n-1})+ \kappa(\mathbb{F}^{1}_{n-2}).\]

By calculation, $\kappa(\mathbb{F}^{1}_1)=2=f_{3}$ and  $\kappa(\mathbb{F}^{1}_2)=3=f_{4}$, we have $\kappa(\mathbb{F}^{1}_n)=f_{n+2}.$
  \item[Alternative proof of $\kappa(\mathbb{F}'^{1}_n)~(n\geq 2)$.] \[\kappa(\mathbb{F}'^{1}_n)=\kappa(\mathbb{F}'^{1}_n\backslash e_n)+ \kappa(\mathbb{F}'^{1}_n/ e_n)=\kappa(\mathbb{F}'^{1}_{n-1})+ \kappa(\mathbb{F}'^{1}_{n-2}\vee P_2),\]
and
\[\kappa(\mathbb{F}'^{1}_{n-2}\vee P_2)=\kappa((\mathbb{F}'^{1}_{n-2}\vee P_2) \backslash e_{n-1})+\kappa( (\mathbb{F}'^{1}_{n-2}\vee P_2) / e_{n-1})=0+\kappa(\mathbb{F}'^{1}_{n-2})=\kappa(\mathbb{F}'^{1}_{n-2}).\]
Hence
\[\kappa(\mathbb{F}'^{1}_n)=\kappa(\mathbb{F}'^{1}_{n-1})+ \kappa(\mathbb{F}'^{1}_{n-2}).\]

By calculation,  $\kappa(\mathbb{F}'^{1}_2)=2=f_2+\ell_1$ and $\kappa(\mathbb{F}'^{1}_3)=5=f_3+\ell_2$,  we see that $\kappa(\mathbb{F}'^{1}_n)=f_{n}+\ell_{n-1}.$
  \item[Alternative proof of $\kappa(\mathbb{F}'^{n}_n)~(n\geq 3)$.]

\[\kappa(\mathbb{F}'^{n}_n)=\kappa(\mathbb{F}'^{n}_n \backslash e_n)+ \kappa(\mathbb{F}'_n/ e_n)=\kappa(\mathbb{F}'_{n-1})+ \kappa(\mathbb{F}'^{n-1}_{n-1}),\]
and
\[\kappa(\mathbb{F}'^{n}_n)=\ell_{n-2}+ \kappa(\mathbb{F}'^{n-1}_{n-1}).\]

By calculation,  $\kappa(\mathbb{F}'^{3}_3)=4=\ell_{3}$, we obtain that $\kappa(\mathbb{F}'^{n}_n)=\ell_{n}.$

\item[Alternative proof of $\kappa(\mathbb{W}^{1}_{n})~(n\geq 3)$.]
\begin{eqnarray*}
&~&\mathbb{W}^{1}_{n}=(-1,n,2,1,3,2,4,3,\cdots,i, i-1,i+1,i,\cdots,n-1,n-2,n,n-1),\\
&~&\mathbb{W}^{1}_{n}\backslash e_{1}=(n,2,3,2,4,3,\cdots,i, i-1,i+1,i,\cdots,n-1,n-2,n,n-1)=\mathbb{F}_{n-1},\\
&~&\mathbb{W}^{1}_{n}/ e_{1}=(-2,-n,3,2,4,3,\cdots,i, i-1,i+1,i,\cdots,n-1,n-2,n,n-1),\\
&~&\mathbb{W}^{1}_{n}/e_{1}\backslash e_{2}=(-n,3,4,3,\cdots,i, i-1,i+1,i,\cdots,n-1,n-2,n,n-1)=\mathbb{F}^{1}_{n-2},\\
&~&\mathbb{W}^{1}_{n}/e_{1}/e_{2}=
(-3,n,4,3,\cdots,i, i-1,i+1,i,\cdots,n-1,n-2,n,n-1)=\mathbb{W}^{1}_{n-2}.
\end{eqnarray*}
We have
\begin{eqnarray*}
\kappa(\mathbb{W}^{1}_{n})&=&\kappa(\mathbb{W}^{1}_{n}\backslash e_{1})+\kappa(\mathbb{W}^{1}_{n}/e_{1})\\
&=&\kappa(\mathbb{W}^{1}_{n}\backslash e_{1})+\kappa(\mathbb{W}^{1}_{n}/e_{1}\backslash e_{2})+\kappa(\mathbb{W}^{1}_{n}/e_{1}/ e_{2})\\
&=&\kappa(\mathbb{F}_{n-1})+\kappa(\mathbb{F}^{1}_{n-2})+
\kappa(\mathbb{W}^{1}_{n-2})\\
&=&f_{n}+f_{n}+\kappa(\mathbb{W}^{1}_{n-2}).
\end{eqnarray*}
Then
\[\kappa(\mathbb{W}^{1}_{n})-\kappa(\mathbb{W}^{1}_{n-2})=2f_{n}, where~ n\geq 5.\]
By calculation, $\kappa(\mathbb{W}^{1}_{3})=6, \kappa(\mathbb{W}^{1}_{4})=8,$
we obtain that $$\kappa(\mathbb{W}^{1}_{n})=2f_{n+1}-1+(-1)^{n+1}.$$
\end{description}

\subsection{Combinatorial interpretations}


\begin{lemma}[\cite{Merino}]\label{perfect1}
	The $(n+1)$-th Fibonacci number equals the number of perfect matchings in $P_2\times P_n$.
\end{lemma}

\begin{figure}[!htbp]
	\centering
	\includegraphics[width=1\linewidth]{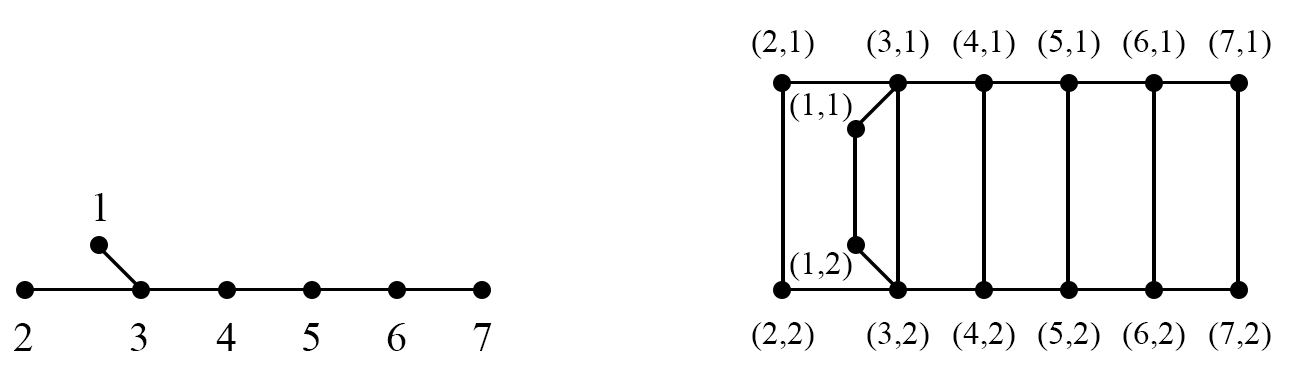}
	\caption{$T_7$ and $P_2\times T_7$}
	\label{fig:caterpillar}
\end{figure}

\begin{theorem}
	The $(n-1)$-th Lucas number equals the number of perfect matchings in $P_2\times T_n$, where $T_n$ is a special caterpillar as shown in Fig \ref{fig:caterpillar}.
\end{theorem}

\begin{proof}
Let $Pf(P_2\times T_n)$ denote the number of perfect matchings in $P_2\times T_n$.
If we choose the edge $((1,1),(1,2))$ as one matching edge, then we need to find out all other $n-1$ matching edges in $P_2\times P_{n-1}$, whose number is $Pf(P_2\times P_{n-1})=f_n$ by Lemma \ref{perfect1}. If we choose edges $((1,1),(3,1))$ and $((1,2),(3,2))$ as two matching edges, then we need to find out all other $n-2$ matching edges in the disjoint between $P_2\times P_{n-3}$ and $P_2$ (with edge $((2,1),(2,2))$), whose number is $Pf(P_2\times P_{n-3})=f_{n-2}$ by Lemma \ref{perfect1}. Hence,
\[Pf(P_2\times T_n)=Pf(P_2\times P_{n-1})+Pf(P_2\times P_{n-3})=f_{n}+f_{n-2}=\ell_{n-1}.\]
\end{proof}

Given an $n \times n$ matrix $A$, let us define $E_{k}(A)$ as the sum of its principal minors of size $k > 0$, and $E_{0}(A)=1$. The following formula for the \emph{characteristic polynomial} of $A$ is well-known, see \cite{Roger}.
\[\det(tI_n-A)=E_{0}(A)t^{n}-E_{1}(A)t^{n-1}+\cdots+(-1)^{n}E_{n}(A).\]

The Fibonacci polynomial is defined as $f_{1}(x)=1$, $f_{2}(x)=x$ and $$f_{n+1}(x)=xf_{n}(x)+f_{n-1}(x).$$
The Lucas polynomial is defined as
$$\ell_{n-1}(x)=f_{n}(x)+f_{n-2}(x).$$
The first few polynomials are:
$\ell_{2}(x)=x^{2}+2$, $\ell_{3}(x)=x^{3}+3x$, $\ell_{4}(x)=x^{4}+4x^{2}+2$ and $\ell_{5}(x)=x^{5}+5x^{3}+5x$.

\begin{lemma}[\cite{Merino}]
The characteristic polynomial of the matrix $A(\mathbb{F}_{n})$ equals $(n+1)$-th Fibonacci
polynomial.
\end{lemma}

\begin{theorem}
	The characteristic polynomials of the matrices $A(\mathbb{F}'_{n}), A(\mathbb{F}'^{1}_{n}), A(\mathbb{F}^{1}_{n}), A(\mathbb{W}^{1}_{n})$ and $A(\mathbb{F}'^{n}_{n})$ are as follows.

\begin{eqnarray*}
\det(tI_n-A)=\left\{\begin{array}{ll}
                    t\ell_{n-1}(t), & \mbox{if}~A=A(\mathbb{F}'_{n}),\\
t\ell_{n-1}(t)-f_{n}(t), & \mbox{if}~A=A(\mathbb{F}'^{1}_{n}),\\
f_{n+1}(t)-f_{n}(t), & \mbox{if}~A=A(\mathbb{F}^{1}_{n}),\\
(t-1)f_{n}(t)+2f_{n-1}(t)+(-1)^{n+1}-1, & \mbox{if}~A=A(\mathbb{W}^{1}_{n}),\\
                    t(\ell_{n-1}(t)-\ell_{n-2}(t)), & \mbox{if}~A=A(\mathbb{F}'^{n}_{n}).
                   \end{array}\right.
\end{eqnarray*}
\end{theorem}

\begin{proof}
Let $M:=tI_{n}-A$, where $A$ is equal to $A(\mathbb{F}'_{n}), A(\mathbb{F}'^{1}_{n}), A(\mathbb{F}^{1}_{n}), A(\mathbb{W}^{1}_{n})$ and $A(\mathbb{F}'^{n}_{n})$, respectively. The result follows by expanding the
determinant of $M$ along the first column as follows.
	\begin{eqnarray*}
		\det(tI_{n}-A(\mathbb{F}'_{n}))&=&t\cdot \det(M[1, 1])+\sum_{i=2}^{n}(-1)^{i+1}m_{i,1}\cdot \det(M[i,1])\\
		&=& t \det(tI_{n-1}-A(\mathbb{F}_{n-1}))+\sum_{i=2}^{n}(-1)^{i+1}m_{i,1}\cdot \det(M[i,1])\\
		&=& tf_{n}(t)+(-1)^{3}m_{1,3}m_{2,2}m_{3,1}\det(tI_{n-3}-A(\mathbb{F}_{n-3}))\\
		&=&tf_{n}(t)+tf_{n-2}(t)\\
		&=&t\ell_{n-1}(t).
	\end{eqnarray*}
\begin{eqnarray*}
		\det(tI_{n}-A(\mathbb{F}'^{1}_{n}))&=&(t-1)\cdot \det(M[1, 1])+\sum_{i=2}^{n}(-1)^{1+i}m_{i,1}\cdot \det(M[i,1])\\
		&=& (t-1) 	\det(tI_{n-1}-A(\mathbb{F}_{n-1}))+(-1)^{3}m_{1,3}m_{2,2}m_{3,1}\det(tI_{n-2}-A(\mathbb{F}_{n-3}))\\
		&=& (t-1)f_{n}(t)+tf_{n-2}(t)\\
		&=&t\ell_{n-1}(t)-f_{n}(t).
	\end{eqnarray*}
	\begin{eqnarray*}
		\det(tI_{n}-A(\mathbb{F}^{1}_{n}))&=&(t-1)\cdot \det(M[1, 1])+\sum_{i=2}^{n}(-1)^{1+i}m_{i,1}\cdot \det(M[i,1])\\
		&=& (t-1) 	\det(tI_{n-1}-A(\mathbb{F}_{n-1}))+(-1)^{3}m_{1,2}m_{2,1}\det(tI_{n-2}-A(\mathbb{F}_{n-2}))\\
		&=& (t-1)f_{n}(t)+f_{n-1}(t)\\
		&=&f_{n+1}(t)-f_{n}(t).
	\end{eqnarray*}
	\begin{eqnarray*}
		\det(tI_{n}-A(\mathbb{W}^{1}_{n}))&=&(t-1)\cdot \det(M[1, 1])+\sum_{i=2}^{n}(-1)^{1+i}m_{i,1}\cdot \det(M[i,1])\\
&=& (t-1) 	\det(tI_{n-1}-A(\mathbb{F}_{n-1}))+(-1)^{1}m_{1,2}m_{2,1}\det(tI_{n-2}-A(\mathbb{F}_{n-2}))\\
& &+(-1)^{n+3}+(-1)^{n+2}[(-1)(-1)^{n-2}+(-1)^{n-1+1}\cdot \det(tI_{n-2}-A(\mathbb{F}_{n-2}))]\\
		&=& (t-1)f_{n}(t)+f_{n-1}(t)+(-1)^{n+1}-1+f_{n-1}(t)\\
		&=&(t-1)f_{n}(t)+2f_{n-1}(t)+(-1)^{n+1}-1.
	\end{eqnarray*}
\begin{eqnarray*}
	\det(tI_{n}-A(\mathbb{F}'^{n}_{n}))&=& t\cdot \det(M[1, 1])+\sum_{i=2}^{n}(-1)^{1+i}m_{i,1}\cdot \det(M[i,1])\\
	&=&   t\cdot \det(tI_{n}-A(\mathbb{F}^{1}_{n-1}))
	+(-1)^{3}m_{1,3}m_{2,2}m_{3,1}\det(tI_{n}-A(\mathbb{F}^{1}_{n-3}))\\
	&=& t(f_{n}(t)-f_{n-1}(t))+t(f_{n-2}(t)-f_{n-3}(t))\\
	&=&t(\ell_{n-1}(t)-\ell_{n-2}(t)).
\end{eqnarray*}
\end{proof}

\section{Concluding remarks}
If a bouquet contains several non-orientable loops, then sometimes it may have a partial dual which is a bouquet with exactly one non-orientable loop. In this case, we can obtain the number of quasi-trees by Matrix-Quasi-tree Theorem $($Theorem \ref{quasitreeTH}$)$.
However, there exists a non-orientable bouquet that does not have a partial-dual which is a bouquet with exactly one non-orientable loop as follows.
\begin{example}
Let $B=(-1,-2,3,1,2,4,3,4)$. Then $D(B)=(E, \mathcal{F})$, where $E=\{1,2,3,4\}$ and $$\mathcal{F}=\{\emptyset, \{1\},\{2\},\{13\}, \{23\},\{34\},\{234\}, \{134\}\}.$$ There are at least two single-element feasible sets if $D(B)$ twists any feasible set. By Proposition \ref{tauandtwist}, there are at least two non-orientable loops in any partial-dual bouquet $B^{\delta(F)}$ where $F\in \mathcal{F}$. Note that if $F\notin \mathcal{F}$, then $B^{\delta(F)}$ is not a bouquet. Furthermore,  we obtain that $\kappa(B)=|\mathcal{F}|=8$, but $\det(I_4+A(B))=14$. Thus Matrix-Quasi-tree Theorem $($Theorem \ref{quasitreeTH}$)$ does not extend to all non-orientable bouquets.
\end{example}

Here are two problems to go for future study.

\begin{problem}
Give the Matrix-Quasi-tree Theorem for all non-orientable bouquets.
\end{problem}

\begin{problem}
Determine the class of ribbon graphs $G$ for which there exists $A\subseteq E(G)$ such that $G^{\delta(A)}$ is a bouquet with exactly one non-orientable loop.
\end{problem}

\section{Acknowledgment}
This work is supported by NSFC (Nos. 12101600, 12171402), and partially supported by the Natural Science Foundation of Hunan Province, P. R. China (No. 2022JJ40418), the Excellent Youth Project of Hunan Provincial Department of Education, P. R. China (No. 23B0117), and the China Scholarships Council (No. 202108430063). The first author thanks the National Institute of Education, Nanyang Technological University, where part of this research was performed.


\begin{thebibliography}{10}

		

\bibitem{bollobas} B. Bollob\'{a}s, O. Riordan, A polynomial of graphs on surfaces, \emph{Math. Ann.}, 323 (2002) 81--96.

\bibitem{AB1} A. Bouchet, Greedy algorithm and symmetric matroids,
\emph{Math. Program.}, {38} (1987) 147--159.	

\bibitem{Bouchet87} A. Bouchet, Unimodularity and circle graphs, \emph{Discrete Math.}, 66 (1987) 203--208.



\bibitem{Chmutov} S. Chmutov, Generalized duality for graphs on surfaces and the signed Bollob\'{a}s-Riordan polynomial, \emph{J. Combin. Theory Ser. B}, 99 (2009), 617--638.

\bibitem{ChunJCTA} C. Chun, I. Moffatt, S. D. Noble, R. Rueckriemen, Matroids, delta-matroids and embedded graphs, \emph{J. Combin. Theory Ser. A}, 167 (2019) 7--59.
	
\bibitem{Chun} C. Chun, I. Moffatt, S. D. Noble, R. Rueckriemen, On the interplay between embedded graphs and delta-matroids, \emph{Proc. London Math. Soc.},
	118 (2019) 675--700.
		
\bibitem{Deng} Q. Deng, X. Jin , M. Metsidik, Characterizations of bipartite and Eulerian partial duals of ribbon graphs, \emph{Discrete Math.}, 343 (2020) 111637.

\bibitem{Deng2} Q. Deng, X. Jin, Q. Yan, Twist polynomial as a weight system for set systems, {arXiv: 2404.10216v1}, 2024.


\bibitem{EM} J. A. Ellis-Monaghan, I. Moffatt, Graphs on Surfaces, Springer, New York, 2013.


\bibitem{GT} J. L. Gross, T. W. Tucker, Topological Graph Theory, John Wiley \&
Sons, New York, 1987.


\bibitem{Haselgrove} C. B. Haselgrove, A note on Fermat's last theorem and Mersenne numbers, \emph{Eureka}, 11 (1949) 19--22.

\bibitem{Roger} R. A. Horn, C. R. Johnson, Matrix Analysis, Cambridge University
Press, Cambridge, 2013.




\bibitem{Lando17} S. Lando, V. Zhukov,  Delta-matroids and Vassiliev invariants, \emph{Mosc. Math. J.}, {17 (2017)} 741--755.

			
\bibitem{Lauri} J. Lauri, On a formula for the number of Euler trails for a class of digraphs, \emph{Discrete Math.}, 163 (1997) 307--312.

\bibitem{Macris} N. Macris, J. V. Pul\'{e}, An alternative formula for the number of Euler trails for a class of digraphs,  \emph{Discrete Math.}, 154 (1996) 301--305.

\bibitem{Merino}  C. Merino,  The number of quasi-trees in fans and wheels,  \emph{Electron. J. Combin.}, 30 (2023) $\sharp$P1.46.

\bibitem{Merino2023} C. Merino, I. Moffatt, S. D. Noble, The critical group of a combinatorial map, arXiv: 2308.13342v1, 2023.



\bibitem{Moffatt15} I. Moffatt, Separability and the genus of a partial dual, \emph{European J. Combin.}, 34 (2013) 355--378.
		
\bibitem{Moffatt} I. Moffatt, E. Mphako-Banda, Handle slides for delta-matroids,
\emph{European J. Combin.}, 59 (2017) 23--33.

\bibitem{Naji} W. Naji, Reconnaissance des graphes de cordes, \emph{Discrete Math.}, 54 (1985) 329--337.




		
\bibitem{Sedl67} J. Sedl\'{a}\v{c}ek, Finite graphs and their spanning trees, \emph{\v{C}asopis P\v{e}st. Mat.}, 92 (1967) 338--342.
	
\bibitem{Sedl69} J. Sedl\'{a}\v{c}ek,  On the number of spanning trees of finite graphs, \emph{\v{C}asopis P\v{e}st. Mat.}, 94 (1969) 217--221.








\end{thebibliography}
\end{document}